\documentclass[12pt]{amsart}
\usepackage{color}
\usepackage[all]{xy}
\usepackage{amssymb,amsmath}
\usepackage{enumerate}
\usepackage{mathrsfs}
\usepackage{eucal}

\usepackage{setspace}

\date{March 21, 2015}

\calclayout
\newtheorem{dummy}{anything}[section]
\newtheorem{theorem}[dummy]{Theorem}
\newtheorem*{thma}{Theorem A}
\newtheorem*{thmb}{Theorem B}

\newtheorem{lemma}[dummy]{Lemma}
\newtheorem{proposition}[dummy]{Proposition}

\theoremstyle{definition}
\newtheorem{definition}[dummy]{Definition}
  \newtheorem{example}[dummy]{Example}

    \newtheorem*{question}{Question}
  \newtheorem*{acknowledgement}{Acknowledgement}

\newcommand{\cF}{\mathcal F}

\newcommand{\cH}{\mathcal H}

\newcommand{\cP}{\mathcal P}

\newcommand{\bC}{\mathbf C}
\newcommand{\bD}{\mathbf D}

\newcommand{\bV}{\mathbf V}

\newcommand{\bbF}{\mathbb F}

\newcommand{\bbQ}{\mathbb Q}

\newcommand{\bbZ}{\mathbb Z}

\newcommand{\Zp}{\bbZ _{(p)}}

\DeclareMathOperator{\Res}{Res}

\DeclareMathOperator{\hdim}{homdim}
  \DeclareMathOperator{\rk}{rank}
 \DeclareMathOperator{\rank}{rank}

\newcommand{\cy}[1]{\bbZ/{#1}}
\newcommand{\la}{\langle}
\newcommand{\ra}{\rangle}

\newcommand{\bd}{\partial}

\def\bZp{\bbZ_{(p)}}

\def\G{\varGamma}

\DeclareMathOperator{\Fix}{Fix}
\DeclareMathOperator{\Or}{Or}

\newcommand\RG{R\G}

\newcommand\un{\underline}

\DeclareMathOperator{\Qd}{Qd}
\newcommand{\bigast}{\divideontimes}

\def\maprt#1{\smash{\,\mathop{\longrightarrow}\limits^{#1}\,}}

\begin{document}

\title[Group actions on spheres with rank one prime power isotropy]
{Group actions on spheres with rank one\\prime power isotropy}\author{Ian Hambleton}
\author{Erg\"un Yal\c c\i n}

\address{Department of Mathematics, McMaster University,
Hamilton, Ontario L8S 4K1, Canada}

\email{hambleton@mcmaster.ca }

\address{Department of Mathematics, Bilkent University,
06800 Bilkent, Ankara, Turkey}

\email{yalcine@fen.bilkent.edu.tr }

\thanks{Research partially supported by NSERC Discovery Grant A4000. The second author was also partially supported by the Scientific and Technological Research Council of Turkey (T\" UB\. ITAK) 
through the research program B\. IDEB-2219.}

\begin{abstract}
We show that a rank two finite group $G$ admits a finite $G$-CW-complex $X\simeq S^n$ with rank one prime power isotropy if and only if $G$ does not $p'$-involve $\Qd(p)$ for any odd prime $p$. This follows from a more general theorem which allows us to 
construct a finite $G$-CW-complex by gluing together a given $G$-invariant family of representations defined on the Sylow subgroups of $G$.
\end{abstract}

\maketitle
 
\section{Introduction}
\label{sect:introduction}

Actions of finite groups on spheres can be studied in various different geometrical settings.  The fundamental examples come from the unit spheres $S(V)$ in a real or complex $G$-representation $V$, and already  natural questions arise for these examples about the dimensions of the non-empty  fixed sets  $S(V)^H$, $H \leq G$, and the structure of the isotropy subgroups. 

A useful way to measure the complexity of the isotropy is the \emph{rank}.
We say that $G$ has \emph{rank $k$} if it contains a subgroup isomorphic to $(\bbZ/p)^k$, for some prime $p$, but no subgroup $(\bbZ/p)^{k+1}$, for any prime $p$. In this paper we answer the following question:

\begin{question}\label{ques:mainquestion} For which finite groups  $G$, does there exist a finite $G$-CW-complex $X\simeq
S^n$ with all isotropy subgroups of rank one ~?
\end{question}

By P.~A.~Smith theory, the rank one assumption on the isotropy subgroups implies that $G$ must have $\rk(G)\leq 2$ (see \cite[Corollary 6.3]{h-yalcin2}). Since every rank one finite group can act freely on a finite
complex homotopy equivalent to a sphere (Swan \cite{swan1}), we can restrict our attention to rank two  groups. 
Here are three  natural settings for the study of finite group actions on spheres:

\begin{enumerate}[(A)]
\item  smooth $G$-actions
  on closed manifolds homotopy equivalent to spheres; 
\item finite $G$-homotopy representations (see tom Dieck 
\cite[Definition 10.1]{tomDieck2});
\item finite $G$-CW-complexes $X \simeq S^n$.
\end{enumerate}

In contrast to  $G$-representation spheres $S(V)$, the non-linear smooth $G$-actions on a smooth manifold $M \simeq S^n$ exhibit more flexibility. For example, in the linear case,  the fixed sets $S(V)^H$ are always linear subspheres. For smooth actions, the fixed sets are smoothly embedded submanifolds but may not even be integral homology spheres. 

Well-known general constraints on smooth actions arise from P.~A.~Smith theory: if $H$ is a subgroup of $p$-power order, for some prime $p$,  then $M^H$ is a $\Zp$-homology sphere. In addition, even if the fixed sets are diffeomorphic to spheres, they may be knotted or linked as embedded subspheres in $M$ (see \cite{Dieck:1985b}, \cite{Davis:1988}).  One can also consider topological $G$-actions, usually with the assumption of local linearity, otherwise the fixed sets may not be locally flat  submanifolds. 

In the setting (B) of $G$-homotopy representations, the objects of study are finite (or more generally finite-dimensional) $G$-CW-complexes $X$ satisfying the property that for each $H\leq G$, the fixed point set $X^H$ is homotopy equivalent to a sphere $S^{n(H)}$ where $n(H)=\dim X^H$. We could also consider a version of this setting where $\dim X^H$  is the same as its homological dimension, and $X^H$ is a $\Zp$-homology $n(H)$-sphere, for $H$ of $p$-power order.

The third setting (C) is the most flexible of all. Here we suppose that $X\simeq S^n$ is a finite $G$-CW-complex homotopy equivalent to a sphere, but do not require that $\dim X = n$. Moreover, we make no initial assumptions about the homology of the fixed sets $X^H$, although the conditions imposed by P.~A.~Smith theory with
 $\bbF_p$-coefficients still hold. In the setting (C), we will see that  $\dim X^H$ must be (much) higher in general than its homological dimension, and this provides new obstructions to understanding our motivating question in setting (A) or (B).

In this paper we provide a complete answer for the existence question in setting (C). Our construction produces $G$-CW-complexes with prime power isotropy.

\begin{thma}\label{thm:main}
Let $G$ be a finite group of rank two. If $G$ admits a finite $G$-CW-complex $X\simeq S^n$ with rank one  isotropy then $G$ is $\Qd(p)$-free. Conversely, if $G$ is $\Qd(p)$-free, then there exists a finite $G$-CW-complex $X\simeq S^n$ with rank one prime power isotropy.
\end{thma}

The group $\Qd(p)$ is defined as the semidirect product 
$$\Qd(p) = (\bbZ/ p \times \bbZ / p)\rtimes SL_2(p)$$
with the obvious action of $SL_2(p)$ on $\bbZ / p \times \bbZ /p$. We say $\Qd(p)$ is $p'$-\emph{involved in $G$} if there exists a subgroup $K \leq G$, of order prime to $p$, such that $N_G(K)/K$ contains a subgroup isomorphic to $\Qd(p)$. If a group $G$ does not $p'$-involve $\Qd(p)$ for any odd prime $p$, then we say that $G$ is \emph{$\Qd(p)$-free}. 

In our earlier work \cite{h-yalcin3} and \cite{h-yalcin2}, we studied this problem  in the setting (B) of 
\emph{$G$-homotopy representations},
introduced by tom Dieck (see \cite[Definition 10.1]{tomDieck2}). 
We found a list of conditions on a rank two finite group $G$ that guarantees the existence of a finite $G$-homotopy representation with rank one prime power isotropy. Identifying the full list of necessary and sufficient conditions is still an open problem, but we did provide a complete answer \cite[Theorem C]{h-yalcin2} for rank two finite simple groups.

The necessity of the $\Qd(p)$-free condition was established in \cite[Theorem 3.3]{unlu1} and \cite[Proposition 5.4]{h-yalcin2}.
In the other direction, if $G$ is a rank two finite group which is $\Qd(p)$-free then  $G$  has a $p$-effective representation $V_p \colon G_p \to U(n)$ (see Definition \ref{def:peffective}) which can be used to construct finite  $G$-CW-complexes $X \simeq S^n$  with rank one isotropy. The existence of these $p$-effective representations 
was proved by Jackson \cite[Theorem~47]{jackson1} and they were also one of the main ingredients for the constructions in Hambleton-Yal{\c c}{\i}n \cite{h-yalcin2}. 

To do the construction in Theorem A, we prove a more technical theorem. We now introduce more terminology to state this theorem. For each prime $p$ dividing the order of $G$, let $G_p$ denote a fixed Sylow $p$-subgroup of $G$. 

\begin{definition}\label{def:Ginvariant} Suppose that we are given a family of \emph{Sylow representations} $\{V_p\}$ defined on Sylow $p$-subgroups $G_p$, over all primes $p$. We say the family $\{ V_p \}$ is \emph{$G$-invariant} if \begin{enumerate} \item  $V_p$ \emph{respect fusion in $G$}, i.e., the character $\chi _p$ of $V_p$ satisfies 
$\chi_p (gxg^{-1})=\chi_p(x)$ whenever $gxg^{-1} \in G_p$ for some $g \in G$ and $x \in G_p$; and \item for 
all $p$, $\dim V_p$ is equal to a fixed positive integer $n$. \end{enumerate}  \end{definition}

Given a $G$-invariant family of Sylow representations $\{ V_p \}$, we construct a $G$-equivariant spherical fibration $q\colon E \to B$ over a contractible $G$-space $B$ with isotropy in $\cP$ such that  for every $x\in \Fix(B, G_p) = B^{G_p}$, the fiber $q^{-1} (x)$ is $G_p$-homotopy equivalent to $S (V_p ^{\oplus k})$ for some $k\geq 1$ (see Theorem \ref{pro:fibconst}). The total space of this $G$-fibration has many interesting properties: in particular, it admits a $G$-map $$f_0 \colon \coprod _p G \times _{G_p} S(V_p ^{\oplus k}) \to E.$$
By adapting the $G$-CW-surgery techniques introduced by Oliver-Petrie \cite{Oliver:1982} to this $G$-map, we obtain a finite  $G$-CW-complex $X \simeq S^{2kn-1}$ whose restriction to Sylow $p$-subgroups resembles the linear spheres $S(V_p^{\oplus k})$. In particular, we prove the following theorem (see Definition \ref{p-localEq} for the definition of $p$-local $G$-equivalence).

\begin{thmb} Let $G$ be a finite group.   Suppose that $\{V_p\colon G_p \to U(n)\}$ is  a $G$-invariant family of Sylow representations. Then there exists a positive integer $k\geq 1$ and a finite $G$-CW-complex $X\simeq S^{2kn-1}$ with prime power isotropy,  such that  the $G_p$-CW-complex $\Res ^G _{G_p} X$ is $p$-locally $G_p$-equivalent to $S(V_p ^{\oplus k} )$, for every prime $p \mid |G|$,
\end{thmb} 

This theorem was stated by Petrie \cite[Theorem C]{petrie1978} in a slightly different form and a sketched proof was provided. Related results were proved by tom Dieck   (see \cite[Satz 2.5]{Dieck:1982},  \cite[Theorem 1.7]{Dieck:1985a}). Although we use some of the steps of these arguments, we believe that a proof of Theorem B does not exist in the literature. All the previous constructions seem to aim towards obtaining a  finite $G$-CW-complex $X\simeq S^{m}$ with $\dim X=m$.  However, we showed in \cite{h-yalcin3} and \cite{h-yalcin2} that there are additional necessary conditions for obtaining such a complex with prime power isotropy. Here is a specific example.

\begin{example} Let $G$ denote the dihedral group of order $2q$, with $q$ an odd prime. Let $V_2$ be a trivial representation of $G_2 = \cy 2$, and let $V_q$ be a free unitary representation of $G_q = \cy q$, such that $\dim V_2 = \dim V_q$. Then Theorem B shows that there exists a finite $G$-CW-complex $X \simeq S^{m}$, with $\Fix(X, G_2) \simeq S^{m}$ ($2$-locally),  and $\Fix(X, G_q) = \emptyset$, for some  integer $m=2kn-1$. However, these conditions imply that $\dim X > m$ by \cite[Proposition 2.10]{h-yalcin3} (compare \cite[Theorem 4.2]{Straume:1981a}).
\end{example}

The paper is organized as follows: In Section \ref{sect:acyclic}, we show that for every finite group $G$, there is a finite-dimensional  contractible $G$-space $B$ with prime power isotropy, such that for every $p$-subgroup $H$, the fixed point set $X^H$ is $\bbZ _{(p)}$-acyclic. This might be of independent interest, since  P.~A.~Smith theory only guarantees that the fixed sets are $\bbF_p$-acyclic.  In Section \ref{sect:fibration}, using this space as base space, we construct a $G$-equivariant fibration $q\colon  E \to B$ with fiber type $S(V_H ^{\oplus k})$, for a given compatible family $\{V_H\}$ of representations.  The total space $E$ has only prime power isotropy and its restriction to $G_p$ is $p$-locally $G_p$-equivalent to $S(V^{\oplus k} _p )$ for some $k\geq 1$. However, $E$ is not a finite $G$-CW complex, and this means that the methods of \cite{Oliver:1982} must be applied with care.

 In Section \ref{sect:plocal}, we prove Proposition \ref{pro:plocal} which allows us to kill homology groups to reach to a $p$-local homotopy equivalence on fixed points of $p$-subgroups.
In Section \ref{sect:main}, we prove our main theorems (Theorem A and Theorem B). Theorem A essentially follows from Theorem B once we apply a theorem of Jackson \cite{jackson1} on the existence of $p$-effective characters for rank two finite groups which are $\Qd(p)$-free.

Finally, we remark that Theorem A was also stated in Jackson \cite[Proposition 48]{jackson1}, but the indication of proof appears to confuse  homotopy actions with finite $G$-CW-complexes. The motivation for Theorem A comes from the work of Adem and Smith \cite{adem-smith1} on the existence of free actions of finite groups on a product of two spheres. There is an interesting set of conditions related to this problem which we discussed in detail in \cite[Section 1]{h-yalcin2}. We refer the reader to this discussion for further details on the history of this problem.

\begin{acknowledgement}  The second author would like to thank McMaster University for the support provided by a H.~L.~Hooker Visiting Fellowship, and the Department of Mathematics \& Statistics at McMaster for its hospitality while this work was done. 
\end{acknowledgement}

\section{Acyclic complexes with prime power isotropy}
\label{sect:acyclic}

The main purpose of this section is to prove the following theorem.

\begin{theorem}\label{thm:contractible}
Let $G$ be a finite group and $\cP$ denote the family of all subgroups of $G$ with prime power order. Then there exists a finite-dimensional contractible $G$-CW-complex $X$, with isotropy in $\cP$, such that for every $p$-subgroup $P \leq G$, the fixed point subspace $X^P$ is $\Zp$-acyclic.
\end{theorem}

There is a similar theorem by Leary and Nucinkis \cite[Proposition 3.1]{Leary:2010} for infinite groups acting on contractible complexes, which implies in particular that for a finite group $G$, there is a finite-dimensional contractible $G$-CW-complex $X$ with isotropy in $\cP$. But this contractible complex  is constructed using a mapping telescope, and the fixed point subspaces  are $\bbF_p$-acyclic but  do not have finitely generated $\Zp$-homology.

Let $\cF_p$ denote the family of all $p$-subgroups of $G$. The family $\cP$ is the union of families $\cF_p$ over all over all primes $p$ dividing the order of $G$. To prove Theorem \ref{thm:contractible}, we  first prove the following result.

\begin{proposition}\label{pro:p-acyclic}
Let $G$ be a finite group and $p$ be a prime such that $p \mid |G|$. Then, there exists a finite-dimensional $G$-CW-complex $X$, with isotropy in $\cF_p$, such that 
for every $p$-subgroup $P \leq G$, the fixed point subspace $X^P$ is $\Zp$-acyclic.
\end{proposition}

A finite-dimensional $\bbF_p$-acyclic complex with $p$-subgroup isotropy is constructed in \cite[Theorem 2.14]{jackowski-mcclure-oliver1}. But this construction also uses a mapping telescope so it does not have finitely generated $\bbZ_{(p)}$-homology.

The construction we propose uses some of our earlier methods for constructing $G$-CW-complexes. In particular, we use chain complexes over the orbit category.  Recall that the orbit category $\Gamma _G:=\Or _{\cH} G$ over a family $\cH$ is the category with objects $G/H$, where $H\in \cH$, and whose morphisms are given by $G$-maps $G/H \to G/K$. Given a commutative ring $R$ with unity, an $R\G _G$-module is defined as contravariant functors from $\G_G$ to the abelian category of $R$-modules. For more details on $R\G_G$-modules we refer the reader to \cite{hpy1} (see also L\"uck \cite[\S 9, \S 17]{lueck3} and tom Dieck \cite[\S 10-11]{tomDieck2}).  

Recall that for every family $\cH$, there is a universal space $E_{\cH} G$ such that isotropy subgroups of $E_{\cH} G$ are in 
$\cH$ and for every $H \in \cH$, the fixed point set $(E_{\cH}G)^H$ is contractible.
If $\bC =\bC (E_{\cH} G ^? ; R)$ denote the cellular chain complex (over the orbit category) of the space $E_{\cH} G$, then $\bC$ is a chain complex of free $R\G_G$-modules. Note that the augmented complex 
$$\widetilde \bC : \quad \cdots \to \bC_n \maprt{\bd_n} \bC_{n-1} \maprt{\bd_{n-1}} \bC_{n-2} \to \cdots \to \bC _1 \to \bC _0 \to \un{R} \to 0,$$
 is an exact sequence, where $\un{R}$ denotes the constant functor. Hence $\bC$ is a projective resolution of $\un{R}$ as an $R\G _G$-module.

\begin{lemma}\label{lem:kernel} Let $\cH=\cF_p$, the family of all $p$-subgroups in $G$, and let $R=\Zp$. Then there is a positive integer $n$ such that $\ker \bd _{n-1}$ is a projective $R\G_G$-module.
\end{lemma}  

\begin{proof} This follows from the fact that $\un{R}$ has a finite projective dimension as an $R\G_G$-module (see \cite[Corollary 3.15]{hpy1}). Note that $n$ can be taken as any integer greater or equal to the homological dimension of $\un{R}$ as an $R\G_G$-module.
\end{proof}
 
Now we are ready to prove Proposition \ref{pro:p-acyclic}.
 
\begin{proof}[Proof of Proposition \ref{pro:p-acyclic}]
Let $\bC =\bC (E_{\cF_p} G ^? ; R)$ and $P=\ker \bd_{n-1}$ denote the projective $R\G_G$-module for a suitably large $n$ (as in Lemma \ref{lem:kernel}). To avoid problems in low dimensions, we also assume $n \geq 3$. Let $Q$ be a projective $R\G_G$-module such that $P\oplus Q $ is a free $\RG_G$-module. Using the Eilenberg swindle, we see that $\ker \bd_n \oplus F \cong F$, where $F=Q \oplus P \oplus Q \oplus \cdots$ is an infinitely generated free $\RG_G$-module. Adding the chain complex $$\cdots \to 0 \to F \maprt{id} F \to 0 \to \cdots $$ to the truncated complex, we obtain a complex of free $\RG_G$-modules
$$ 0 \to F \maprt{\varphi} \bC_{n-1} \oplus F \maprt{(\bd_{n-1}, 0)} \bC_{n-2} \to \cdots \to \bC_1 \to \bC_0 \to 0$$
where the map $\varphi$ is defined as the composition $F \cong \ker \bd_{n-1} \oplus F \hookrightarrow \bC_{n-1} \oplus F$.
Note that this chain complex can be lifted to a chain complex of free $\bbZ \G_G$-modules 
$$ \bD : 0 \to \bD_{n} \to \bD_{n-1} \to \bD_{n-2} \to \cdots \to \bD _0 \to 0$$
where the resulting complex has homology groups that are (possibly infinitely generated) abelian groups with torsion coprime to $p$. Because of the special structure of the original $\RG_G$-complex, we can assume that the lifting $\bD$ is of the form
$$ \bD: 0 \to \widehat F \maprt{\widehat \varphi} \widehat \bC_{n-1}  \oplus \widehat F \maprt{(\bd_{n-1}, 0)}\widehat \bC_{n-2} \to \cdots \to \widehat \bC_1 \to \widehat \bC_0 \to 0$$
where $\widehat \bC_i=\bC_i (E_{\cF_p } G ; \bbZ )$ and $\widehat F$ is  a free $\bbZ \G_G$-module such that $\widehat F \otimes R \cong F$. 

The map $\widehat \varphi$ is obtained as follows: let $\{e_i\}$ be a basis for $F$ as an $R\G_G$-module. For each $i$, there is an integer $s_i$, coprime to $p$, such that $\varphi( s_i e_i)\in  \widehat \bC_{n-1}  \oplus \widehat F$. Let $\widehat F$ be the $\bbZ \G_G$-submodule of $F$ generated by $\{ s_i e_i\}$ and $\widehat \varphi$ be the map induced by $\varphi$. It is easy to see from this that the reduced homology of this complex $\bD$ is zero except at dimension $n-1$ and  $H_{n-1} (\bD)$ is a torsion abelian group with torsion coprime to $p$ (possibly infinitely generated). 

Note that we can assume that $\bD$ is partially realized by the $(n-1)$-skeleton of the complex $E_{\cF_p } G$. In fact, by attaching orbits of cells to $E_{\cF_p} G$ with $p$-subgroup isotropy, we can assume that $\bD$ is realized for dimensions $\leq n-1$. The last realization step can be done using \cite[Lemma 8.1]{hpy1}. Note that for this step we need to assume $n\geq 3$. 

Hence, we can conclude that for every finite group $G$, there is a finite-dimensional $G$-CW-complex $X$ with isotropy in $\cF_p$, such that \begin{enumerate}
\item $X$ is $n$-dimensional and $(n-2)$-connected where $n=\max \{ 3, \hdim \un{R}\}$; \item for each $P \in \cF_p$, the only nontrivial reduced homology of the fixed point subspace $X^P$ is at dimension $n-1$ and $H_{n-1}(X)$ is a torsion abelian group with torsion coprime to $p$. \end{enumerate}
In particular, for every $P \in \cF_p$, the fixed point subspace $X^P$ is $\Zp$-acyclic. Hence this completes the proof of Proposition \ref{pro:p-acyclic}.
\end{proof}

\begin{proof}[Proof of Theorem \ref{thm:contractible}]
In Proposition \ref{pro:p-acyclic}
we have constructed  a $\Zp$-acyclic complex $X_p$ of dimension $n_p$,  for each $p \mid |G|$. Let $X$ be the join $\bigast X_p$ of all the $X_p$'s over all $p \mid |G|$. The reduced homology of $X$ is nonzero only at dimension $n-1$, where $n=\prod n_p$, and
 $$H_{n-1} (X) \cong \bigotimes _{p \mid |G|} H_{n_p-1} (X_p).$$ 
 Since $H_{n_p-1}( X_p)$ is a torsion group coprime to $p$, the homology group $H_{n-1}(X)$ is a torsion abelian group with torsion coprime to $|G|$. Such an abelian group has two step free resolution. To see this, note that as a $\bbZ G$-module $N=H_{n-1} (X)$ is cohomologically trivial since it is a torsion group with torsion coprime to the order of the group. If we take a free cover of $N$, then we get an exact sequence of the form $$0 \to M \to F_0 \to N \to 0.$$ Note that the module $M$ is both torsion free and cohomologically trivial. Hence by \cite[Theorem 8.10, p.~152]{brown1}, $M$ is a projective module. By an Eilenberg swindle argument, we can add free modules to $M$ and $F_0$ to obtain a two step free resolution for $N$. This means, we can kill the last homology group at dimension $n-1$ by adding free orbits of cells. By taking further joins if necessary, we an assume that $X$ is simply connected, hence the resulting $G$-CW-complex is contractible. For each $1 \neq P \in \cF_p$, we have $H_* (X^P ; \Zp) \cong H_* (X_p ^P; \Zp)\cong H_* (pt; \Zp)$, so the fixed subspace $X^P$ is $\Zp$-acyclic for every $P \in \cF_p$.
\end{proof}

\section{$G$-equivariant fibrations}
\label{sect:fibration}

Let $G$ be a finite group. In this section, we first give some necessary definitions related to $G$-fibrations and then construct a $G$-fibration over a contractible base space with prime power isotropy.
For more details on this material we refer the reader to 
\cite[Section 2]{unlu-yalcin3} and to some earlier references mentioned in that paper.

\begin{definition}\label{defn:Gfibration} A $G$-fibration is a $G$-map $q\colon E \to B$ which satisfies the following homotopy lifting property for every $G$-space $X$: given a commuting diagram of $G$-maps 
$$\xymatrix{
X \times \{ 0\} \ar[d] \ar[r]^-{h}
& E \ar[d]^{q}  \\
X \times I \ar[r]^-{H}
&  B,}\\ $$
there exists a $G$-map $\widetilde H\colon  X\times I \to E$ such that $\widetilde H |_{X\times \{0\}}=h$ and $p \circ \widetilde H =H$.
\end{definition}

If $p \colon E\to B$ is a $G$-fibration, then for every $x\in B$, the isotropy subgroup $G_x \leq G$  acts on the fiber space $F_x=q^{-1} (x)$. So, $F_x$ is a $G_x$-space. 

\begin{definition}\label{def:fibertype} Let $\cH$ be a family of subgroups of $G$ and $\{ F_H \}$ denote a family of $H$-spaces over all $H \in \cH$. If for every $x \in B$, the isotropy subgroup $G_x$ lies in $\cH$ and the fiber space $F_x$ is $G_x$-homotopy equivalent to $F_{G_x}$, then $p \colon E \to B$ is said to have {\it fiber type} $\{ F_H \}$. 
\end{definition}

Here and throughout the paper \emph{a family of subgroups} always means a collection of subgroups which are closed under conjugation and taking subgroups. In general a $G$-fibration does not have to satisfy the above criteria: for $x, y \in B$ with $G_{x}=G_{y}=H$, it may happen that $F_{x}$ and $F_{y}$ are not $H$-homotopy equivalent. Throughout the paper we only consider $G$-fibrations which do have a fiber type. 

We will construct $G$-equivariant spherical fibrations whose fiber type is given by a family of linear $G$-spheres. To start we assume that we are given a compatible family of representations.

\begin{definition} Let $\cH$ be a family of subgroups of $G$ and $\bV=\{ V_H \}$ denote a family of complex $H$-representations defined over $H \in \cH$. We say $\bV$ is a \emph{compatible family of representations} if $f^* (V_K) \cong V_H$ for every $G$-map $f\colon  G/H \to G/K$. In this case, we call $\bV$ an $\cH$-representation (see \cite[Definition 3.1]{h-yalcin2}).
\end{definition}

Note that since $1\in \cH$, all the $H$-representations $V_H$ in $\bV$ have the same dimension. We call this common dimension the dimension of $\bV$. We have the following result as a main tool for constructions of $G$-fibrations which was first proved by 
Klaus  \cite[Proposition 2.7]{klaus1}].

\begin{theorem}\label{pro:fibconst} Let $G$ be a finite group, with $\cH$ a family of subgroups. 
Let $B$ be a finite-dimensional $G$-CW-complex such that  the isotropy subgroup $G_x$ lies in $\cH$, for every $x\in B$. Given a compatible family of complex representations $\bV=\{ V_H\}$ defined over $\cH$, there exists an integer $k \geq 1$ and a $G$-equivariant spherical fibration $q \colon E\to B$ such that the fiber type of $q$ is $\{S(V_H ^{\oplus k} ) \}$.
\end{theorem}

\begin{proof} See \cite[Proposition 4.3]{unlu-yalcin3}.
\end{proof}

We will apply this theorem to construct a $G$-fibration over a base space with prime power isotropy. As before, let $\cP$ denote the family of all subgroups of $G$ with prime power order, and $\cF_p$ denote the family of all $p$-subgroups of $G$.  

\begin{lemma}\label{lem:fibertype} Let $G$ be a finite group and $\{ V_p \}$ be a $G$-invariant family of Sylow  representations (see Definition \textup{\ref{def:Ginvariant}}). For each $H \in \cF_p$, let $V_H$ be the representation obtained from $V_p$ via the map $$ H \maprt{c^g} gHg^{-1} \hookrightarrow G_p$$
where $c^g$ denotes the conjugation map $h \mapsto ghg^{-1}$ and the second map is the inclusion map (the element $g\in G$ is chosen arbitrarily such that $gHg^{-1} \leq G_p$). Then the collection $\bV=(V_H)_{H \in \cP}$ is a compatible family of representations over $\cP$.
\end{lemma}

\begin{proof} We only need to check that when $H, K \leq G_p$ are
such that $H=gKg^{-1}$ for some $g\in G$, then $(c^g)^* (V_H )\cong V_K$ as $K$-representations. Note that the isomorphism holds because for every $x \in K$, we have $$(c^g)^* (\chi _p)(x)=\chi_p(gxg^{-1})=\chi _p (x)$$ by the character formula given in Definition \ref{def:Ginvariant}. This also shows that the compatible family $\{ V_H \}$ does not depend on the elements $g \in G$ chosen to define it (up to isomorphism).
\end{proof}

Suppose that we are given  a $G$-invariant family of Sylow representations $\{ V_p \}$. Then by Lemma \ref{lem:fibertype}, this gives a compatible family of representations $\bV=(V_H)$. Let $B$ be the $G$-CW-complex constructed in Proposition \ref{thm:contractible}. By applying Proposition \ref{pro:fibconst} to the base space $B$ with family $\bV$, we obtain a $G$-equivariant spherical fibration $q \colon  E \to B$  with fiber type $\{ S(V_H ^{\oplus k} )\}_{H \in \cP} $ for some $k \geq 1$. 

The total space $E$ satisfies the certain properties which will be used in our construction of finite homotopy $G$-spheres.  

\begin{definition}\label{p-localEq}
A $G$-map $f \colon  X \to Y$ between two $G$-spaces is called a \emph{$p$-local $G$-equivalence} if for every subgroup $H \leq G$, the map on fixed point sets $f^H\colon X^H \to Y^H$ induces an isomorphism on $\bbZ_{(p)}$-homology.

We say that two $G$-spaces $X$ and $Y$ are \emph{$p$-locally $G$-equivalent} if for some $k$ there is are $G$-spaces $\{X_i\}$ and $\{Y_i\}$, for $0\leq i \leq k$, such that $X_0 =X$ and $Y_k = Y$, together with two families of $G$-maps $ X_i \to Y_i$ for $i \geq 0$,  and $ X_i \to Y_{i-1}$ for $i >0$, which are $p$-local $G$-equivalences. 
\end{definition}
  
Now we prove the main result of this section.

\begin{proposition}\label{pro:targetspace} Let $G$ be a finite group, and let  $\{ V_p \}$  be a $G$-invariant family of Sylow  representations. Then 
there exists an integer $k\geq 1$ and a finite-dimensional $G$-CW-complex $E$, with isotropy in $ \cP$, satisfying the following properties:
\begin{enumerate}
\item $E$ is homotopy equivalent to a sphere $S^{2kn-1}$ where $n=\dim V_p$;
\item For every $H \in \cP$, the fixed point subspace $E^H$ is simply connected;
\item  For every $p \mid |G|$, there is a $G_p$-map
$j_p\colon S(V_p ^{\oplus k} ) \to E$ which is a $p$-local $G_p$-equivalence.
\end{enumerate}
\end{proposition}

\begin{proof} Let $B$ be a contractible $G$-CW-complex as in 
Theorem \ref{thm:contractible}, 
and $E$ be the total space of a fibration $q\colon E \to B$ with fiber type $\{ S(V_H ^{\oplus k} )\}_{H \in \cP} $ for some $k \geq 1$. By construction of the $G$-fibration,  the total space 
$E$ is a $G$-homotopy equivalent to a finite-dimensional $G$-CW-complex (see \cite[Proposition 4.4]{unlu-yalcin3}). Since $B$ has isotropy in $\cP$, the total space $E$  has isotropy in $ \cP$. Since $B$ is contractible, $E$ is homotopy equivalent to $S^{2kn-1}$.

For every $H \leq G$, the induced map $q^H\colon  E^H \to B^H$ on fixed subspaces is a fibration with fiber type $F^H$. We can assume that for every $P \in \cF_p$, the fixed point subspace $B^H$ is simply connected (if not we can replace $B$ with $B*B$). We can also assume that the subspaces $F^H$ are simply connected by replacing $k$ with a larger integer if necessary. Using the long exact homotopy sequence for the fibration $F^H \to E^H \to B^H$, we obtain that $E^H$ is simply connected for every $H \in \cP$. 

For second statement, observe that for every $p \mid |G|$, the fixed point space $B^{G_p}$ is non-empty, by 
P.~A.~Smith Theory. If we take $x\in B^{G_p}$, then the inclusion map $i_x\colon  \{x\} \to B^{G_p}$ induces a $G_p$-map $j_x\colon F_x \to E$, where $F_x=q^{-1}(x)$. By the definition of fiber type, we have $F_x \simeq S(V_p ^{\oplus k})$ as a $G_p$-space. We define $j_p$ as the composite $S(V_p ^{\oplus k})\simeq F_x \maprt{j_x} E$ which is a $G_p$-map. For each subgroup $H \leq G_p$, we have a fibration diagram:
$$\xymatrix{F_x ^H \ar[d] \ar@{=}[r]&   F_x ^H \ar[d] \\
F_x ^H \ar[d]\ar[r]^{j_x ^H } & E^H \ar[d] \\ 
\{x \} \ar[r]^{i_x^H}& B^H.}$$ 
Since $i_x ^H$ induces a $\bbZ_{(p)}$-homology isomorphism, the map $j_x ^H$ also induces a $\bbZ_{(p)}$-homology isomorphism. This can be seen easily by a spectral sequence argument. Note that $B^H$ is simply connected, so the $E_2$-term of the Serre spectral sequence for the second fibration is of the form 
$E_2 ^{i,j}=H^i (B^H; H^j (F_x^H, \bbZ _{(p)}))$ with untwisted coefficients. 
By comparing two spectral sequences, we see that $j_x^H$ induces an isomorphism on $\bbZ_{(p)}$-homology. This shows that $j_p$ is a $p$-local $G_p$-equivalence. 
\end{proof}

\section{$p$-local $G$-CW-surgery}
\label{sect:plocal}

Let $G$ be a finite group, $\cP$ denote the family of subgroups of $G$ with prime power order, and $\{ V_p\}$ be a $G$-invariant family of Sylow representations $V_p\colon G_p \to U(n)$ over all primes $p$ dividing the order of $G$. In Section \ref{sect:fibration}, we proved that there is a finite-dimensional $G$-CW-complex $E$, with isotropy in $\cP$, homotopy equivalent to $S^{2kn-1}$ for some $k\geq 1$,  satisfying some further fixed point properties. 

To prove Theorem B we will need to replace $E$ with a finite $G$-CW-complex $X$ having properties similar to $E$, with possibly a larger $k\geq 1$. We will do this by applying the $G$-CW-surgery techniques introduced in \cite{Oliver:1982} to a particular $G$-map (see also \cite{Dieck:1982a}).

By part (iii) of Proposition \ref{pro:targetspace},  there is a $G_p$-map
$ j_p \colon  S(V_p ^{\oplus k} )\to E$ which induces a $\bbZ_{(p)}$-homology isomorphism on fixed subspaces,
 for every $p \mid |G|$. Using these maps we can define a $G$-map $$f_0\colon \coprod_{p \mid |G|} G \times _{G_p } S(V_p ^{\oplus k} )\to E$$ by taking $f_0 (g, x)=gj_p (x)$ for every $g \in G$ and $x \in S(V_p ^{\oplus k} )$. It is clear that $f_0$ is well-defined and it is a $G$-map, where the $G$-action on $G\times _{G_p} S(V_p ^{\oplus k} )$ is by left multiplication. We will apply $G$-CW-surgery methods to this map to convert it to a homotopy equivalence. 

The first step of this surgery method is to get a $p$-local homology equivalence on $H$-fixed subspaces for every nontrivial $p$-subgroup $H\leq G$. We will do this step-by-step by a downward induction starting 
from  Sylow $p$-subgroups. At a particular step $H$ we will need to attach cells to complete that step. The following proposition is the main result of this section and it states exactly what we will need to complete a particular step in the downward induction.

\begin{proposition}\label{pro:plocal}
Let $G$ be a finite group and $f \colon  X \to Y$ be a $G$-map between two simply connected $G$-CW-complexes, with isotropy subgroups in $\cF_p$, such that 
\begin{enumerate}
\item $X$ is a finite complex and $X^P$ is an odd-dimensional $\bZp$-homology sphere for every $p$-subgroup $1\neq P \leq G$;
\item $Y$ is a finite-dimensional complex with finitely generated $\bbZ_{(p)}$-homology; 
\item  The Euler characteristic $\sum _i \dim _{\bbQ} (-1)^i [H_i (Y; \bbQ)] =0 \in R_{\bbQ}(G)$, the rational representation ring of $G$.
\end{enumerate}
If for every $p$-subgroup $1\neq P \leq G$, the induced map $f^P\colon X^P\to Y^P$ on fixed point sets is a $\bbZ_{(p)}$-homology equivalence, then by attaching finitely many free $G$-orbits of cells to $X$, we can extend $f$ to a
$\bbZ_{(p)}$-homology equivalence $f' \colon  X' \to Y$. 
\end{proposition}

Given a $G$-map $f \colon  X \to Y$ between two $G$-CW-complexes, we define the $n$-th homotopy group of $f$, denoted by $\pi_n(f)$, as the equivalence classes of pairs of maps $(\alpha, \beta)$ such that the diagram
$$\xymatrix{S^{n-1}\ar[d]^{\alpha} \ar[r]^{i} &  D^n\ar[d]^{\beta} \\
X \ar[r]^{f} & Y}$$ 
commutes, where $i: S^{n-1} \to D^n$ is the inclusion map of the boundary of $D^n$. The equivalence relation is given by a pair of homotopy that fits into a similar diagram. It is easy to show that
$\pi _n (f)$ isomorphic to the $n$-th homotopy group of the pair $\pi _n (Z_f, X)$,
where $Z_f$ denotes the mapping cylinder $(X\times I) \cup_f Y$. We consider $X$ as a subspace by identifying $X$ with $X\times \{0\}$.

In a similar way, we can define relative homology group of a $G$-map $f \colon  X\to Y$ in coefficients in $R$ as follows:
$$K_* (f; R):=H_* (Z_f , X ;R) \cong \widetilde H_* (M_f; R),$$
following the notation in \cite{Oliver:1982}, 
where $M_f$ denotes the mapping cone of $f$. We recall the relative Hurewicz theorem for homotopy groups of pairs. 

\begin{lemma}\label{lem:Hurewicz}
Let $R=\bbZ$ or $\bbZ_{(p)}$ for some prime $p$, and let $f \colon  X \to Y$ be a map between two simply connected spaces. For $n \geq 2$, if $\pi_i(f)\otimes R =0$ for all $i<n$, then $K_n(f; R) =0$ for all $i< n$ and the Hurewicz map $\pi _n (f)\otimes R \to K_n (f; R)$ is an isomorphism.
\end{lemma}

\begin{proof} See \cite[Theorem 7.5.4]{spanier}.
\end{proof}

The Hurewicz theorem allows us to realize homology classes as homotopy classes. We kill the corresponding homotopy class by attaching free orbits of cells to $X$ and extending the map $f$. If the homotopy class is represented by a pair of maps $(\alpha, \beta)$ as above, then the space $X'$ is defined as the space $X'=X\cup _{\alpha} D^n$ and the map $f' \colon X' \to Y$ is defined by $$f'(x)= \begin{cases} f(x) & \text{if $x\in X$} \\ \beta(x) & \text{if $x\in D^n$} 
\end{cases} 
$$ In the homotopy group $\pi _n (f')$, the homotopy class for the pair $(\alpha, \beta)$ is now zero and this cell attachment does not introduce any more homotopy classes at dimensions $i\leq n$.

Let $f \colon  X \to Y$ be a $G$-map as in Proposition \ref{pro:plocal}. By applying this cell attachment method we can assume that  $f$ is extended to a map $f_1 \colon  X_1 \to Y$ such that $d:=\dim X_1 > \dim Y$ and $f_1$ induces an $\Zp$-homology isomorphism in dimensions $i< d$. Since $Y$ has finitely generated $\bbZ_{(p)}$-homology, in the process only finitely many free $G$-orbits are attached to $X$. So $X_1$ is still a finite complex. 

Note that $K_i (f_1 ; \Zp)$ is nonzero only at dimension $i=d+1$, and $$M:=K_{d+1} (f_1; \Zp) \cong H_d (X_1 ; \Zp).$$ Since $X_1$ is a finite complex and $d=\dim X_1$, as a $\Zp$-module $M$ is a finitely generated and torsion free. We claim that $M$ is a free $\Zp G$-module. First we prove a lemma which shows, in particular, that $M$ is projective.

\begin{lemma}\label{lem:projective} Let $R=\bbZ$ or $\Zp$, and let $f \colon  X \to Y$ be a $G$-map such that $d:=\dim X > \dim Y$ and $f$ induces an $R$-homology isomorphism on dimensions $i< d$. Assume also that for every $1 \neq H \leq G$, the induced map $f^H\colon X^H\to Y^H$ on fixed point subspaces is an $R$-homology equivalence. 
Then $K_{d+1} (f; R)$ is a projective $RG$-module.
\end{lemma}

\begin{proof} Let $X^s=\cup _{1\neq H \leq G} X^H$ and $f^s \colon  X^s \to Y^s$ denote the the restriction of $f$ to the singular set. For every nontrivial subgroup $H \leq G$, the induced map $f^H\colon X^H\to Y^H$ is an $R$-homology equivalence. This gives, in particular, that $f^s\colon X^s \to Y^s$ is an $R$-homology equivalence.  Let $Z_{f} ^s := X \cup Z_{f^s }$. Consider the homology sequence for the triple $(Z_{f}, Z_{f}^s, X)$ with coefficients in $R$:
$$\cdots \to H_i( Z_{f} ^s  , X) \to H_i (Z_{f} , X) \to H_i (Z_{f} , Z_{f}^s ) \to H_{i-1} (Z_{f} ^s , X) \to \cdots $$
We have $$H_i( Z_{f} ^s  , X )=H_i (X \cup Z_{f ^s } , X)\cong H_i (Z_{f^s}, X^s ) =0 $$ for all $i$, because $f^s$ is an $R$-homology equivalence.
From this we obtain that $H_i(Z_{f}, Z_{f}^s ) \cong H_i (Z_{f}, X)$, hence  $H_i (Z_{f},  Z_{f}^s ; R) =0$ for $i< d+1$ and it is isomorphic to $K_{d+1} (f; R)$ when $i=d+1$. 

The chain complex for the pair $(Z_{f}, Z_{f}^s)$ in $R$-coefficients gives an exact sequence of $RG$-modules 
$$ 0 \to K _{d+1} (f; R) \to C_{d+1}(Z_{f}, Z_{f}^s ; R)\to \cdots \to C_0(Z_{f}, Z_{f}^s ; R)\to 0.$$ 
For all $i$, the modules $C_i (Z_{f}, Z_{f}^s ; R)$ are free $RG$-modules, hence this exact sequence splits and $K_{d+1} (f; R)$ is a projective $RG$-module.
\end{proof}

Applying this lemma to the map  $f_1\colon X_1 \to Y$ constructed above, we obtain that $M=K_{d+1} (f_1; \Zp)\cong H_d (X_1: \Zp)$ is a projective $\Zp G$-module. Now we show that $M$ is a free $\Zp G$-module.

\begin{lemma}\label{lem:free} Let $f \colon  X \to Y$ be a $G$-map as in Proposition \textup{\ref{pro:plocal}} and $f_1\colon X_1 \to Y$ is the map obtained by attaching cells to $X$ as above. Then, $K_{d+1} (f_1;  \Zp)$ is a finitely-generated free $\Zp G$-module. 
\end{lemma}

\begin{proof} By Lemma \ref{lem:projective}, $M=K_{d+1} (f_1; \Zp)$ is a projective 
$\Zp G$-module. Let $\bbQ M=\bbQ \otimes M$. By \cite[Lemma 2.4]{Oliver:1982}, $M$ is a free $\Zp G$-module if $\chi _{\bbQ M} (g) =0$ for all $1\neq g \in G$. Since $M \cong H_d (X_1 ; \Zp)$ and $X_1$ is a finite $G$-CW-complex, we can calculate $\chi _{\bbQ M}$ using the the chain complex of $X_1$. Let $$ 0 \to C_{d} (X_1; \bbQ)  \to C_{d-1} (X_1; \bbQ)  \to \cdots \to C_0 (X_1; \bbQ) \to 0$$ be the chain complex for $X_1$ in $\bbQ$-coefficients. In rational representation ring of $G$, we have 
$$(-1)^{d}[H_{d} (X_1 ; \bbQ) ]+ \sum _{i=0} ^{d-1}  (-1) ^{i} [H_{i} (X_1; \bbQ)]= 
\sum _{i=1} ^{d}(-1)^i [C_i (X_1; \bbQ)]$$ 
Since $f_1$ induces $\Zp$-homology isomorphism at dimensions $i<d$, we get
$$\sum _{i=0} ^{d-1}  (-1) ^{i} [H_{i} (X_1; \bbQ)] =\sum _{i=0} ^{d-1}  (-1) ^{i} [H_{i} (Y; \bbQ)] =0$$
by the assumption in Proposition \ref{pro:plocal}.
This gives that for every $1\neq g \in G$, 
$$(-1)^{d} \chi _{\bbQ M} (g) = \sum _{i=1}^{d}(-1)^i \dim _{\bbQ} C_i (X_1^g ; \bbQ) =
\sum _{i=1}^{d} (-1)^i \dim _{\bbQ} H_i (X_1^g ; \bbQ)=\chi (X_1 ^{\la g \ra})$$
Since for every $p$-group $1\neq H\leq G$, the fixed point set $X_1 ^H$ is an odd dimensional $\bbZ_{(p)}$-homology sphere, we have $\chi (X_1 ^H )=0$ for every nontrivial $p$-subgroup $H \leq G$.  When $1\neq H \leq G$ is a $p'$-subgroup, then $X_1 ^H$ is empty, so again the Euler characteristic is zero. Hence $\chi _{\bbQ M} (g) =0$ for all $1\neq g \in G$. We conclude that $M$ is a free $\Zp G$-module.  
\end{proof}

\begin{proof}[Proof of Proposition \ref{pro:plocal}] Let $f \colon  X \to Y$ be a $G$-map as in Proposition \ref{pro:plocal}, and let $f_1\colon X_1 \to Y$ be the $G$-map obtained by attaching cells, as described above,  so that $f_1$
 induces an $\Zp$-homology isomorphism in dimensions $i< d$. By Lemma \ref{lem:free},   $M=K_{d+1} (f_1;  \Zp)$ is a finitely-generated free $\Zp G$-module. 
 By Lemma \ref{lem:Hurewicz},  
 $$ K_{d+1} (f_1; \Zp) \cong \pi_{d+1}(f_1)\otimes \Zp , $$
  and hence $\pi_{d+1}(f_1)$ contains a finitely-generated free $\bbZ G$-module $M'\subseteq \pi_{d+1}(f_1)$ with index prime to $p$. We attach free orbits of 
$G$-cells to $X_1$ using the pairs of maps $(\alpha, \beta)$ representing homotopy classes of $\bbZ G$-basis elements in $M'$. The resulting map $f'\colon X' \to Y$ is a $\Zp$-homology equivalence.
\end{proof}

\section{Proof of main theorems}
\label{sect:main}

In this section we prove Theorem A and Theorem B  as stated in the introduction. Theorem A will follow from Theorem B almost directly by applying a theorem by Jackson \cite[Theorem 47]{jackson1}.

Let $G$ be a finite group, $\cP$ denote the family of all subgroups of $G$ with prime power order. Suppose we are given a $G$-invariant family of Sylow representations $\{V_p\}$ over the primes dividing the order of $G$. We will construct a finite $G$-CW-complex $X\simeq S^{2kn-1}$ such that for every $p \mid |G|$, the restriction of $X$ to $G_p$ is $p$-locally $G_p$-equivalent to $S(V_p ^{\oplus k} )$, for some $k\geq 1$. We showed in Section \ref{sect:plocal} that there is a $G$-map $f_0\colon X_0 \to E$ where $$X_0=\coprod _{p \mid |G|} G\times _{G_p} S(V_p ^{\oplus k} ) $$ and $E$ is the total space of the fibration constructed in Section \ref{sect:fibration}. The $G$-map $f_0$ is induced from the $G_p$-maps $j_p\colon S(V_p ^{\oplus k} ) \to E$ which were introduced in 
Proposition \ref{pro:targetspace}. 

We will first show that by a downward induction and by attaching cells at each step, we can extend the map $f_0$ to a map $f_1 \colon  X_1 \to E$ such that $f_1^H\colon X_1^H \to E^H$ is a $p$-local homology equivalence for every nontrivial $p$-subgroup $H \leq G$. Since we work with unitary representations, 
 the fixed point subspace $E^H$ is an odd dimensional sphere with trivial $N_G(H)/H$-action. 
This implies in particular that as an $N_G(H)/H$-space the fixed point subspace $E^H$ satisfies the Euler characteristic condition for the target space in Proposition \ref{pro:plocal}.

To show that each step of the downward induction can be performed, suppose $H$ is a nontrivial $p$-subgroup such that $f_1 ^K$ is a $p$-local homology equivalence for every $K$ with $|K|>|H|$. Consider the induced
$N_G(H)/H$ action on $X_1 ^H$. By Proposition \ref{pro:plocal}, we can add free $N_G(H)/H$-orbits of cells to $X_1^H$ to extend $f_1 ^H$ to a $p$-local homology equivalence. In fact, by adding cells of orbit type $G/H$ (instead of just $N_H (H)/H$-orbits) to $X_1$ we can make $X_1^L$ a mod-$p$
equivalence for every $L\leq G$ conjugate to $G$. This shows that downward induction can be carried out until we reach to a map $f_1\colon X_1 \to E$ such that $f_1^H $ is a $p$-local homology equivalence for every nontrivial $p$-subgroup $H \leq G$, for all the primes $p \mid |G|$. 

As we did in the previous section, we can add free cells to $X_1$ and extend $f_1$ to a map $f_2: X_2 \to E$ such that  $f_2$ induces a homotopy equivalence for dimensions $i< d$ where $d:=\dim X_2 > \dim E$. 

\begin{lemma}\label{lem:Integralproj} The module $\bbZ G$-module  $M:=K_{d+1} (f_2)\cong H_d (X_2 ,\bbZ )$ is a finitely-generated projective module.
\end{lemma}

\begin{proof} It is enough to show that for every $p \mid |G|$, the $\Zp G_p$-module $\Res^G _{G_p} M\otimes \Zp$ is projective. This follows from Lemma \ref{lem:projective}. 
\end{proof}

In general, $M$ does not have to be a free $\bbZ G$-module, but we will obtain this condition by taking further joins.
To describe the obstructions for finiteness, we need to introduce more definitions.

\begin{definition} Let $X$ be a finite $G$-CW-complex which has integral  homology of an $m$-dimensional (orientable) sphere for $i\leq m$ and for each $i\geq m+1$, assume that $H_i (X, \bbZ)$ is a projective $\bbZ G$-module. Then we say $X$ is a \emph{$G$-resolution} of an $m$-sphere.
\end{definition}

Let $\widetilde K_0 (\bbZ G)$ denote the Grothendieck ring of finitely generated projective $\bbZ G$-modules, modulo finitely generated free modules. 
We define the finiteness obstruction of $G$-resolution of an $m$-sphere as follows:

\begin{definition} Let $X$ be a $G$-resolution of an $m$-sphere. The finiteness obstruction of $X$ is defined as an element in $\widetilde K_0(X)$ as follows:
$$ \sigma (X)=\sum _{i=m+1} ^{\dim X} (-1)^i[H_i (X)] \in \widetilde K_0(\bbZ G).$$
\end{definition}

We have the following observation:

\begin{lemma}\label{lem:join}
Let $X_1$ and $X_2$ be $G$-resolutions of spheres of dimensions $m_1-1$ and $m_2-1$. Then the join space $X_1 \ast X_2$ is a resolution of a sphere of dimension $m_1+m_2-1$. Moreover, we have 
$\sigma (X_1 \ast X_2)=(-1)^{m_2}\sigma (X_1)+(-1)^{m_1}\sigma (X_2)$.   
\end{lemma}

\begin{proof} Since tensor product (over $\bbZ$)  of a projective module with any torsion-free  $\bbZ G$-module is projective, it is easy to see that all the homology above the dimension $m_1+m_2-1$ will be projective. So, $X_1\ast X_2$ is a $G$-resolution. Moreover, the tensor product of any two finitely generated projective $\bbZ G$-modules is stably free as a $\bbZ G$-module (See \cite[Proposition 7.7]{hpy1}). So the only homology groups that contribute nontrivially to $\sigma (X_1\ast X_2)$ will be the homology modules of the form $H_i (X_1) \otimes H_{m_2-1} (X_2 )$, with $i \geq m_1$, or of the form $H_{m_1 -1} (X_1) \otimes H_{i} (X_2)$, with $i \geq m_2$.  
\end{proof}

By Swan \cite[Prop. 9.1]{swan1a}, the obstruction group $\widetilde K_0(\bbZ G)$ is a finite abelian group, so we can apply the above lemma to conclude that there is a positive integer $l$, such that $\sigma (\bigast _l X_2)=0$. Note that $f_2$ induces a $G$-map $\bigast _l f_2\colon \bigast _l X_2 \to \bigast _l E $. We need the following result to complete the proof of Theorem B.

\begin{lemma}\label{lem:realization} Let $X$ be a $G$-resolution of an 
$(m-1)$-dimensional sphere and let $f \colon  X \to  E$ be a $G$-map which induces a homotopy equivalence in dimensions $\leq m-1$. If $\sigma (X)=0$ in $\widetilde K_0(\bbZ G)$,  then by adding finitely many free cells to $X$, the $G$-map $f$ can be extended to a $G$-map $f' \colon  X' \to E$ which induces an isomorphism on homology.
\end{lemma}

\begin{proof} By adding free cells to $X$ above dimension $m-1$, we can assume we have a map $f_1 \colon  X_1 \to E$ such that all the homology of $X_1$ is concentrated at $d=\dim X_1 > m-1$.  Then, it is easy to see that $(-1)^d[H_d(X_1)]=\sigma (X_1)=0$, hence $H_d(X_1)$ is stably free. By adding free cells to $X_1$ at dimension $d$ and $d-1$, we can kill all the remaining homology and extend $f$ to a $G$-map $f' \colon  X' \to E$ which induces an isomorphism on homology. 
\end{proof}

\begin{proof}[Proof of Theorem B] Starting from the map $f_0\colon X_0 \to E$, we first apply $p$-local surgery methods to get a map $f_1 \colon  X_1 \to E$ which induced a $p$-local homology equivalence on fixed points $X_1 ^H \to E^H$ for every nontrivial $p$-subgroup $H\leq G$. This is done by a downward induction as described above. Then we add free orbits of cells to $X_1$ to obtain a map $f_2\colon X_2 \to E$ where $X_2$ is a $G$-resolution. Here we use Lemma \ref{lem:Integralproj} to conclude that $X_2$ is indeed a $G$-resolution. Finally we use Lemma \ref{lem:join} and \ref{lem:realization} to kill the remaining homology by taking further joins. 

As a result of the above construction we obtain a finite $G$-CW-complex $X$ and a $G$-map $f \colon  X \to \bigast _l E$ which induces a homotopy equivalence. Since $\bigast _l E \simeq S^{2kln-1}$, it follows that $X$ is homotopy equivalent to a sphere of dimension $2kln-1$.
For every $p \mid |G|$, we have $G_p$-maps 
$X \to E$ and $S(V_p ^{\oplus lk} ) \to E$ which induce  $p$-local homology equivalences on fixed points. So $\Res ^G _{G_p} X$ and $S(V_p ^{\oplus lk})$ are $p$-locally $G_p$-equivalent.
\end{proof}

Before giving a proof for Theorem A, we recall the following definition.

\begin{definition}\label{def:peffective} 
A finite group $G$ has a $p$-\emph{effective representation} if it has a representation $V_p \colon G_p \to U(n)$ which respects fusion (see Definition\ref{def:Ginvariant}) and  satisfies $\la V_p |_E, 1_E\ra = 0$ for each elementary abelian $p$-subgroup $E\leq G$ with $\rank E = \rank_p G$. \end{definition}
  
\begin{proof}[Proof of Theorem A]
Let $G$ be a finite group of rank two which is $\Qd(p)$-free. By Jackson \cite[Theorem 47]{jackson1}, for each $p \mid |G|$, there is a $p$-effective representation $V_p$. By taking multiples of these representations if necessary, we can assume that they have a common dimension. This gives a $G$-invariant family $\{ V_p\}$ such that $\la V_p |_E , 1_E \ra =0$ for every elementary abelian $p$-subgroup $E\leq G$ with $\rk E=2$. Applying Theorem B to this $G$-invariant family, we obtain a finite $G$-CW-complex  $X$ homotopy equivalent to a sphere $S^{2kn-1}$, for some $k\geq 1$, such that $\Res^G _{G_p} X$ is $p$-locally $G_p$-equivalent to $S(V_p ^{\oplus k} )$, for every $p \mid |G|$. In particular, for every $p$-subgroup $H \leq G$, the fixed point space $X^H$ has the same $p$-local homological dimension as the fixed point sphere $S(V_p  ^{\oplus k})^H$.  Since $S(V_p) ^E= \emptyset$, we have $S(V_p  ^{\oplus k})^H = \emptyset$ for every subgroup $H \leq G$ with $\rk (H)=2$. Hence the isotropy subgroups of $X$ are all rank one subgroups with prime power order.
\end{proof}


\providecommand{\bysame}{\leavevmode\hbox to3em{\hrulefill}\thinspace}
\providecommand{\MR}{\relax\ifhmode\unskip\space\fi MR }
\providecommand{\MRhref}[2]{%
  \href{http://www.ams.org/mathscinet-getitem?mr=#1}{#2}
}
\providecommand{\href}[2]{#2}

\end{document}